\documentclass{amsart}
\usepackage{times, pdfsync}

\usepackage{amsthm}
\usepackage{amssymb}
\usepackage{amsmath}
\usepackage{amscd}

\newcommand{\bg}{\begin{equation}}
\newcommand{\ed}{\end{equation}}
\newcommand{\bga}{\begin{eqnarray}}
\newcommand{\eda}{\end{eqnarray}}

\def\cbdu{\par{\raggedleft$\Box$\par}}

\newtheorem {Theorem}  {Theorem}

\numberwithin{Theorem}{section}

\newtheorem {Lemma}[Theorem]  {Lemma}

\theoremstyle{definition}

\theoremstyle{remark}
\newtheorem{Remark}[Theorem]{\bf Remark}

\newtheorem {Corollary}[Theorem]{\bf Corollary}

\newcounter{remark}
\newcommand{\n}{\nabla}

\renewcommand{\a}{\alpha}

\renewcommand{\div}{\mbox{div}}

\newcommand{\Bb}{{\mathcal B}}

\def  \T   {{\mathbb T}}

\def  \12  {{\frac{1}{2}}}

\def\build#1_#2^#3{\mathrel{\mathop{\kern 0pt#1}\limits_{#2}^{#3}}}

\begin{document}

\title[Norm Inflation in Besov Spaces ]{Norm inflation for generalized magneto-hydrodynamic system}


\author [Alexey Cheskidov]{Alexey Cheskidov}
\address{Department of Mathematics, Stat. and Comp.Sci.,  University of Illinois Chicago, Chicago, IL 60607,USA}
\email{acheskid@math.uic.edu} 
\author [Mimi Dai]{Mimi Dai}
\address{Department of Mathematics, Stat. and Comp.Sci.,  University of Illinois Chicago, Chicago, IL 60607,USA}
\email{mdai@uic.edu} 

\thanks{The work of Alexey Cheskidov was partially supported by NSF Grant
DMS--1108864.}





\begin{abstract}
We consider the  three-dimensional  incompressible magneto-hydrodynamic system with fractional powers of the Laplacian.
We discover a wide range of spaces where the norm inflation occurs and hence
small initial data results are out of reach. The norm inflation
occurs not only in scaling invariant (critical) spaces, but also in supercritical, and, surprisingly, subcritical ones.

\bigskip


\hspace{0.02cm}CLASSIFICATION CODE: 76D03, 35Q35.
\end{abstract}

\maketitle

\section{Introduction}

In this paper we study the three dimensional incompressible magneto-hydrodynamic system
with fractional powers of the Laplacian:
\begin{equation}\label{FMHD}
\begin{split}
&u_t +(u\cdot\n) u-(b\cdot\n) b+\n p=-\nu(-\Delta)^{\a_1} u,\\
&b_t +(u\cdot\n) b-(b\cdot\n) u=-\mu(-\Delta)^{\a_2} b,\\
& u(x, 0)=u_0, \ \ \ b(x, 0)=b_0 
\end{split}
\end{equation}
where $x\in\mathbb{R}^3$, $t\geq 0$, $u$ represents the fluid velocity, $p$
is the pressure, and $b$
is the magnetic field. The parameter $\nu$ denotes the kinematic viscosity coefficient of the fluid, $\mu$ denotes the reciprocal of the magnetic Reynolds number,  and $u_0$, $b_0$ are given divergence free vector fields in $L^2$.  In the case $\a_1=\a_2=1$, equations \eqref{FMHD} reduce to the classical magneto-hydrodynamic system.
This system has attracted a considerable attention, see \cite{CW, DQS, ST, Wu03, Wu11, Wu08} and references therein.

Solutions to the generalized magneto-hydrodynamic system (\ref{FMHD}) are scaling invariant when $\a_1=\a_2=\a>0$. In this case
\begin{align}\notag
u_\lambda(x,t) &=\lambda^{2\a-1} u(\lambda x, \lambda^{2\a}t), \ p_\lambda(x,t) =\lambda^{2(2\a-1)}p(\lambda x, \lambda^{2\a}t) \\
\ b_\lambda(x,t) &=\lambda^{2\a-1} b(\lambda x, \lambda^{2\a}t) \notag
\end{align}
solves (\ref{FMHD}) with the initial data
\begin{equation}
u_{0\lambda} =\lambda^{2\a-1} u_0(\lambda x), \ b_{0\lambda} =\lambda^{2\a-1} b_0(\lambda x), \notag
\end{equation}
provided $(u(x, t), p(x,t), b(x,t))$ satisfies (\ref{FMHD}) with the initial data $(u_0(x), b_0(x))$.
Spaces  invariant under such scaling are critical.  When $\a_1=\a_2=\a$,
the largest critical space for equations (\ref{FMHD})  is $\dot B_{\infty,\infty}^{1-2\a}$ (see \cite{Ca}).

In the case $b=0$, equations (\ref{FMHD}) reduce to the generalized Navier-Stokes system, which has been studied extensively.
Since the global regularity is only known for $\a_1 \geq 5/4$, the question of global well-posedness in various critical
spaces is of great interest. In the classical case $\a_1 = 1$ the best small initial data result is due to Koch and Tataru. In
\cite{KT}, they established the global well-posedness of the Navier-Stokes equations with small initial data in $BMO^{-1}$.
In the case $\a_1 \in (1/2, 1)$ the situation is better as Yu and Zhai \cite{YZ}
showed global well-posedness for small initial data in the largest critical space $\dot B_{\infty,\infty}^{1-2\a_1}$, which is out of reach for $\a_1 \in [1, 5/4)$. Obstacles here are illustrated by
some illposedeness results: Bourgain and Pavlovi\'{c} \cite{BP} showed the
norm inflation for the classical Navier-Stokes equations in $\dot{B}^{-1}_{\infty,\infty}$ in case $\a_1=1$,  
Cheskidov and Shvydkoy \cite{CS, CSnew} proved the
existence of discontinuous Leray-Hopf solutions of the Navier-Stokes equations in  $\dot{B}^{1-2\a_1}_{\infty,\infty}$ with arbitrarily small initial data for $\a_1 \in [1, 5/4)$, and Yoneda \cite{Yo} showed the norm inflation for the classical Navier-Stokes equation in a generalized Besov space near $BMO^{-1}$.

Recently, Cheskidov and Dai \cite{CD} showed the norm inflation in subcrtitical spaces $\dot{B}^{-s}_{\infty,\infty}$ for all
$s \geq \a_1$, $\a_1 \geq 1$. This provides a wide range of spaces where a small initial data result is not expected. Note that the natural space for the norm inflation
$B^{-\a_1}_{\infty,\infty}$ is only scailing invariant in the classical case $\a_1=1$, and
is subcritical for $\a_1>1$. This explains why small initial data results are only available for $\a_1<1$.

The goal of this paper is to find natural norm inflation spaces for the generalized MHD system \eqref{FMHD} and show that in general they are not scailing invariant, even in the classical case.
Namely, we prove that
\begin{Theorem}\label{Mthm}
Let $\a_1,\a_2 \geq 1$. Assume $\theta_1+\theta_2=2\a_2$  for $\theta_1, \theta_2>0$. For any $\delta >0$ there exists
a smooth space-periodic solution $(u(t),b(t))$ of (\ref{FMHD}) with period $2\pi$, such that
\begin{equation}\notag
\|u(0)\|_{\dot{B}^{-\theta_1}_{\infty, \infty}}+\|b(0)\|_{\dot{B}^{-\theta_2}_{\infty, \infty}} \lesssim \delta, 
\end{equation}
and
\begin{equation}\notag
\|b(T)\|_{\dot{B}^{-s}_{\infty, \infty}}\gtrsim \frac{1}{\delta},
\end{equation}
for some $0<T<\delta$ and all $s>0$.
\end{Theorem}
Since for periodic functions the homogeneous and non-homogeneous Besov norms are equivalent, it follows that
\[
\|b(0)\|_{\dot{B}^{-s}_{\infty, \infty}} \lesssim \|b(0)\|_{B^{-s}_{\infty, \infty}}  \lesssim
\|b(0)\|_{\dot{B}^{-\theta_2}_{\infty, \infty}} \lesssim \delta  \qquad \forall s \geq \theta_2.
\]  
This means that the norm inflation for $b$ occurs in all the spaces $\dot{B}^{-s}_{\infty, \infty}$, $s >0$, which is reflected in this corollary.
\begin{Corollary}\label{cor}
Let $\a_1, \a_2 \geq 1$. For any  $s >0$, $\delta >0$, and positive $\theta$ with $\theta \geq 2\a_2-s$, there exists
a smooth space-periodic solution $(u(t),b(t))$ of (\ref{FMHD}), such that
\begin{equation}\notag
\|u(0)\|_{\dot{B}^{-\theta}_{\infty, \infty}} + \|b(0)\|_{\dot{B}^{-s}_{\infty, \infty}} \lesssim \delta,
\end{equation}
and
\begin{equation}\notag
\|b(T)\|_{\dot{B}^{-s}_{\infty, \infty}}\gtrsim \frac{1}{\delta},
\end{equation}
for some $0<T<\delta$.
\end{Corollary}

In a previous work of Dai, Qing and Schonbek \cite{DQS}, the authors obtained norm inflations for the classical MHD system, i.e., $\a_1=\a_2=1$ in (\ref{FMHD}). More precisely, the main result reads that initial data can be constructed such that
\[\|u(0)\|_{\dot{B}^{-1}_{\infty, \infty}} + \|b(0)\|_{\dot{B}^{-1}_{\infty, \infty}} \lesssim \delta, \] 
while 
\[\|b(T)\|_{\dot{B}^{-1}_{\infty, \infty}}\gtrsim \frac{1}{\delta}, \ \qquad \mbox { for some } 0<T<\delta.\]
One can see the above result is a particular case of Corollary \ref{cor} corresponding to $\a_1=\a_2=\theta=s=1$. In fact, taking $\a_1=\a_2=1$, Corollary \ref{cor} guarantees that the magnetic field $b$ develops norm inflation in all the spaces $\dot{B}^{-s}_{\infty, \infty}$, $s>0$,
not only the critical one $\dot{B}^{1-2\a_1}_{\infty, \infty}$.

We point out another difference with \cite{DQS}.  For the classical MHD system, the
proof of the norm inflation is based on the continuity of the bilinear operator (see (2.7) in \cite{DQS}) associated with the heat kernel on certain Koch-Tataru space to control the nonlinear interactions in the error terms.  For the bilinear operator corresponding to the fractional heat kernel, as discussed in \cite{CD}, the continuity property on a modified Koch-Tataru adapted space is not available.
Instead, the bilinear operator is estimated in $L^\infty$ in order to control the error terms in Subsection \ref{sec:yz}.

The rest of the paper is structured in the following way. In Section \ref{sec:pre} we introduce notations  and present some auxiliary results; in Section \ref{sec:interaction} we describe how the diffusions of plane waves interact; Theorem \ref{Mthm} is proved in Section \ref{sec:proof}.

\bigskip

\section{Preliminaries}
\label{sec:pre}

\subsection{Notations}
We denote by $A\lesssim B$ an estimate of the form $A\leq C B$ with
some absolute constant $C$, and by $A\sim B$ an estimate of the form $C_1
B\leq A\leq C_2 B$ with some absolute constants $C_1$, $C_2$.
We also write $\|\cdot\|_{p}=\|\cdot\|_{L^p}$.

\medskip

\subsection{Besov spaces}
Norms in the homogeneous and non-homogeneous Besov spaces 
$\dot{B}_{\infty, \infty}^{-s}$ and $B_{\infty, \infty}^{-s}$ can be defined as follows (see \cite{L}):
\begin{align}\notag
&\|f\|_{\dot{B}_{\infty, \infty}^{-s}}=\sup_{t>0} t^{\frac{s}{2\a}}
\|e^{-t(-\Delta)^\a}f\|_{L^\infty},\\
&\|f\|_{B_{\infty, \infty}^{-s}} = \sup_{0<t<1}t^{\frac{s}{2\a}}
\|e^{-t(-\Delta)^\a}f\|_{L^\infty}\notag.
\end{align}

The homogeneous and non-homogeneous norms are equivalent for periodic functions (see \cite{SchT}). 
It is then easy to observe that
\bg\label{ineq:imbed}
\|f\|_{\dot B_{\infty, \infty}^{-s}(\T^3)} \leq \|f\|_{L^\infty(\T^3)},
\ed
due to the fact that $\|e^{-t(-\Delta)^{\a}}f\|_{L^\infty} \leq \|f\|_{L^\infty}$.

\medskip

\subsection{Bilinear operator}
Let $\mathbb{P}$ denote the projection on divergence-free vector
fields, which acts on a function $\phi$ as
\begin{equation}\notag
\mathbb{P}(\phi)=\phi+\n\cdot(-\triangle)^{-1}\div\phi.
\end{equation}
For $\a>1$ we consider the following bilinear operator:
\begin{equation}\label{Ba}
\Bb_\a (u, v)=\int_{0}^{t} e^{-(t-\tau)(-\Delta)^\a}
\mathbb{P}\n\cdot (u\otimes v)\, d\tau .
\end{equation}
The following estimate holds for the bilinear operator $\Bb_\a$  (see \cite{CD}). 
\begin{Lemma}\label{le:bil} Let $u,v\in L^1(0,T;L^\infty)$ be such that $u\otimes v\in L^1(0,T;L^\infty)$. Then for all $\a>0$, the bilinear operator $\Bb_\a$ satisfies 
\begin{equation}\label{Bac}
\|\Bb_\a (u, v)\|_{\infty}\lesssim \int_0^t\frac{1}{(t-\tau)^{1/{2\a}}} \|u(\tau)\|_{\infty}\|v(\tau)\|_{\infty}d\tau.
\end{equation}
\end{Lemma}

\bigskip

\section{Interactions of plane waves}
\label{sec:interaction}

\subsection{The first iteration approximation of a mild solution}
\label{sec:rewrite}
We decompose a solution $(u,b)$ of (\ref{FMHD}) as
\begin{align}
u&=e^{-t(-\Delta)^{\a_1}}u_0-u_1+y, \label{u}\\
b&=e^{-t(-\Delta)^{\a_2}}b_0-b_1+z, \label{b}
\end{align}
where
\begin{equation} \label{u1}
\begin{split}
u_1(x, t)= &\Bb_{\a_1}(e^{-t(-\Delta)^{\a_1}}u_0(x),e^{-t(-\Delta)^{\a_1}}u_0(x))\\
&-\Bb_{\a_1}(e^{-t(-\Delta)^{\a_1}}b_0(x),e^{-t(-\Delta)^{\a_1}}b_0(x)),
\end{split}
\end{equation}
\begin{equation} \label{b1}
\begin{split}
b_1(x, t)= &\Bb_{\a_2}(e^{-t(-\Delta)^{\a_2}}u_0(x),e^{-t(-\Delta)^{\a_2}}b_0(x))\\
&-\Bb_{\a_2}(e^{-t(-\Delta)^{\a_2}}b_0(x),e^{-t(-\Delta)^{\a_2}}u_0(x)).
\end{split}
\end{equation}
It is easy to obtain that
\begin{equation} \label{y}
y(t) = -\int\limits_{0}^{t}
e^{-(t-\tau) (-\Delta)^{\a_1}} [G_0 (\tau)+G_1 (\tau)+G_2 (\tau)]d\tau,
\end{equation}
\begin{equation} \label{z}
z(t) = -\int\limits_{0}^{t}
e^{-(t-\tau) (-\Delta)^{\a_2}} [K_0 (\tau)+K_1 (\tau)+K_2 (\tau)]d\tau,
\end{equation}
where
\begin{equation}\label{eq:g}
\begin{split}
G_0=&\mathbb{P}[(e^{-t(-\Delta)^{\a_1}} u_0\cdot\n)u_1+(u_1\cdot\n)e^{-t(-\Delta)^{\a_1}} u_0+(u_1\cdot\n)u_1]\\
&-\mathbb{P}[(e^{-t(-\Delta)^{\a_1}} b_0\cdot\n)b_1+(b_1\cdot\n)e^{-t(-\Delta)^{\a_1}} b_0+(b_1\cdot\n)b_1],\\
G_1=&\mathbb{P}[(e^{-t(-\Delta)^{\a_1}} u_0\cdot\n)y+(u_1\cdot\n)y+(y\cdot\n)e^{-t(-\Delta)^{\a_1}} u_0+(y\cdot\n)u_1]\\
&-\mathbb{P}[(e^{-t(-\Delta)^{\a_1}} b_0\cdot\n)z+(b_1\cdot\n)z+(z\cdot\n)e^{-t(-\Delta)^{\a_1}} b_0+(z\cdot\n)b_1],\\
G_2=&\mathbb{P}[(y\cdot\n)y]-\mathbb{P}[(z\cdot\n)z],
\end{split}
\end{equation}
and 
\begin{equation}\label{eq:k}
\begin{split}
K_0=&\mathbb{P}[(e^{-t(-\Delta)^{\a_2}} u_0\cdot\n)b_1+(u_1\cdot\n)e^{-t(-\Delta)^{\a_2}} b_0+(u_1\cdot\n)b_1]\\
&-\mathbb{P}[(e^{-t(-\Delta)^{\a_2}} b_0\cdot\n)u_1+(b_1\cdot\n)e^{-t(-\Delta)^{\a_2}} u_0+(b_1\cdot\n)u_1],\\
K_1=&\mathbb{P}[(e^{-t(-\Delta)^{\a_2}} u_0\cdot\n)z+(u_1\cdot\n)z+(y\cdot\n)e^{-t(-\Delta)^{\a_2}} b_0+(y\cdot\n)b_1]\\
&-\mathbb{P}[(e^{-t(-\Delta)^{\a_2}} b_0\cdot\n)y+(b_1\cdot\n)y+(z\cdot\n)e^{-t(-\Delta)^{\a_2}} u_0+(z\cdot\n)u_1],\\
K_2=&\mathbb{P}[(y\cdot\n)z]-\mathbb{P}[(z\cdot\n)y].
\end{split}
\end{equation}

Note that the terms $G_0, K_0$ do not depend on $y$ or $z$, the terms $G_1, K_1$ are linear,
and the terms $G_2, K_2$ are quadratic in  $y$ and $z$.

\begin{Remark}
Note that  although the second equation in system (\ref{FMHD}) has no pressure, since $u$ and $b$ are both divergence free,  the term $u\cdot\n b-b\cdot\n u$ is automatically divergence free. Hence the projector $\mathbb{P}$ acting on this term does not change the second equation and hence we can write $b_1$ and $K_i$'s as  described above.
\end{Remark}

\subsection{Diffusion of a plane wave }
Here we discuss the diffusions of plane waves.
Let $k\in \mathbb{R}^3$, $v \in\mathbb{S}^2$ and $k \cdot v=0$. Define
\[
u_0 = v\cos(k\cdot x).
\]
Note that 
$ \n\cdot u_0 = 0$ and
\begin{equation}\label{difu}
e^{-t(-\Delta)^\a}v\cos(k\cdot x)=e^{-|k|^{2\a}t}v\cos(k\cdot x).
\end{equation}

We also remark that 
\[
\|v\cos(k\cdot x)\|_{\dot{B}^{-s}_{\infty,\infty}} \sim |k|^{-s},
\]
for $s>0$.

\subsection{Interaction of plane waves}
\label{sec:int}

The interaction between two different single plane waves, essential for the norm inflation phenomenon,  is discussed here. Let $k_i\in\mathbb{R}^3$, $v_i \in\mathbb{S}^2$ and $k_i\cdot v_i=0$, for $i=1, 2$. Define
\begin{equation}\notag
u_{1}= \cos(k_1\cdot x)v_1, \ \ u_{2}=\cos(k_2\cdot x)v_2.
\end{equation}
Assuming  that  $k_2\cdot v_1 = \frac 12$ it is easy to show that
\[
\aligned &\Bb_\a(e^{-t(- \Delta)^\a} u_{1}, e^{-t(- \Delta)^\a} u_{2}) \\
& = \frac 14  v_1\sin((k_2-k_1)\cdot x) \int_0^t
e^{-(|k_1|^{2\a}+|k_2|^{2\a})\tau}e^{-|k_2-k_1|^{2\a}(t-\tau)}d\tau \notag\\ & + \frac
14 v_1 \sin((k_1+k_2)\cdot x)\int_0^t
e^{-(|k_1|^{2\a}+|k_2|^{2\a})\tau}e^{-|k_1+k_2|^{2\a}(t-\tau)}d\tau.\endaligned
\]

So one can see that the interaction between two plane waves is small 
in $\dot{B}^{-s}_{\infty, \infty}$ provided neither the sum nor the difference
of their wave vectors is small in magnitude. On the other hand, the
interaction is sizable provided either the
sum or the difference of the corresponding wave vectors is small in magnitude.

\bigskip

\section{Proof of the main result}
\label{sec:proof}
Here, following  ideas from \cite{BP, CD}, we  construct  initial data that produce norm inflation for \eqref{FMHD}.
As we saw in the previous section, a large number of waves is needed for this phenomenon to occur. 
The constructed initial data will be smooth and  space-periodic. Moreover, we will show that the solution will
remain smooth up to the time of the norm inflation.

\subsection{Initial data}

Let
\begin{eqnarray}\label{u0}
&u_0=r^{-{\beta_1}}\sum_{i=1}^r|k_i|^{\theta_1}v \cos(k_i\cdot x),\\
&b_0=r^{-{\beta_2}}\sum_{i=1}^r|k'_i|^{\theta_2}v' \cos(k'_i\cdot x),\notag
\end{eqnarray}
where $\beta_1, \beta_2, \theta_1, \theta_2>0$ will be determined later. 
We choose

\bigskip
\noindent
{\em Wave vectors:} Let $\zeta=(1,0,0)$ and $\eta=(0,0,1)$. We take the  wave vectors $k_i\in\mathbb{Z}^3$ to be parallel to $\zeta$
with magnitudes 
\begin{align}\label{kl}
|k_i|=2^{i-1}K, \ \ \ \ i=1, 2, 3, ..., r,
\end{align}
where $K$ is a large integer that depends on $r$. We also define
\begin{equation} \label{ks-ks'}
k'_i=k_i+\eta.
\end{equation}

\bigskip
\noindent
{\em Amplitude vectors:} Define 
\bg\label{eq:v}
v=(0,0,1), \ \ \ v'=(0,1,0). 
\ed
We clearly have
\begin{equation}\notag
k_i\cdot v=k'_i\cdot v'=0,
\end{equation}
i.e., the initial data is divergence free.

The initial data construction here is not the same as the one used for the classical MHD system in \cite{DQS}. We still use the fact that the high-high interactions produce low frequency wave, which is the essential part to develop the norm inflation in Besov spaces with negative indexes. However, thanks to the coupling of the velocity field and the magnetic field, we have the flexibility of having more admissible values of $\beta_1, \beta_2, \theta_1$ and  $\theta_2$, which allows us to obtain a wider range of spaces where the system (\ref{FMHD}) develops norm inflation (see Theorem \ref{Mthm} and Corollary \ref{cor}).

We recall the following results whose proofs can be found in \cite{CD}. 

\begin{Lemma}\label{k}Let $\gamma, \theta> 0$. Assume (\ref{kl})-(\ref{eq:v}) are satisfied. Then
\begin{equation}\label{vk0}
k_i \cdot v' = 0,\ \ \ \ k'_i\cdot v =1, \quad \forall
\quad i = 1,2, \dots, r,
\end{equation}
\begin{equation}\label{kll1}
\sum_{j<i}|k_j|^{\theta}\sim |k_{i-1}|^{\theta} \quad\text{and } \ \sum_{j<i}|k_j'|^{\theta}\sim
|k_{i-1}'|^{\theta},
\end{equation}
\begin{equation}\label{ksum2}
\sum_{i=1}^r |k_i|^{\gamma}e^{-|k_i|^{2\a}t} \lesssim t^{-\frac{\gamma}{2\a}}\quad
\text{and } \ \sum_{i=1}^r |k_i'|^{\gamma}e^{-|k_i'|^{2\a}t} \lesssim t^{-\frac{\gamma}{2\a}} .
\end{equation}
\end{Lemma}

Using this lemma we can easily estimate the initial data in Besov spaces.
\begin{Lemma} \label{le:u0}
Let $(u_0, b_0)$ be as in (\ref{u0}) and $\theta_1, \theta_2>0$. Then
\begin{equation}\label{norm:u0-1}
\|u_0\|_{\dot{B}_{\infty, \infty}^{-{\theta_1}}}\lesssim r^{-{\beta_1}}, \qquad \|b_0\|_{\dot{B}_{\infty, \infty}^{-{\theta_2}}}\lesssim r^{-{\beta_2}}.
\end{equation}
\end{Lemma}
\begin{proof}
Due to (\ref{difu}), we have  
\begin{equation} \label{h-u0}
e^{-t(-\Delta)^{\theta_1}}u_0=r^{-{\beta_1}}\sum_{i=1}^r|k_i|^{\theta_1}v
\cos(k_i\cdot x)e^{-|k_i|^{2{\theta_1}}t}.
\end{equation}
Thanks to Lemma \ref{k}, 
\bg \label{first-bound}
\|u_0\|_{\dot{B}_{\infty, \infty}^{-{\theta_1}}} \sim r^{-{\beta_1}}
\sup_{0<t<1} t^{\frac{1}{2}}\sum_{i=1}^r |k_i|^{\theta_1}e^{-|k_i|^{2{\theta_1}}t} \lesssim r^{-{\beta_1}}\notag.
\ed
The estimate of $b_0$ in $\dot{B}_{\infty, \infty}^{-{\theta_2}}$ can be obtained similarly.
\end{proof}
\begin{Lemma}\label{le:u0infty} Let $(u_0, b_0)$ be as in (\ref{u0}) and $\a>0$. Then
\begin{equation}\notag
\|e^{-t(-\Delta)^{\a}}u_0\|_{\infty} \lesssim r^{-{\beta_1}}t^{-\frac{\theta_1}{2\a}}, \ \ \|e^{-t(-\Delta)^{\a}}b_0\|_{\infty} \lesssim r^{-{\beta_2}}t^{-\frac{\theta_2}{2\a}}.
\end{equation}
\end{Lemma}
\begin{proof}
Equality (\ref{h-u0}) and Lemma \ref{k} yield
\bg\notag
\|e^{-t(-\Delta)^{\a}}u_0\|_{\infty} \lesssim r^{-{\beta_1}}
 \sum_{i=1}^r |k_i|^{\theta_1}e^{-|k_i|^{2{\a}}t}
 \lesssim r^{-{\beta_1}}t^{-\frac{\theta_1}{2\a}}.
\ed
The estimate for $b_0$ can be obtained similarly.
\end{proof}

\bigskip

\subsection{Analysis of $u_1$ and $b_1$}
\label{sec:u1b1}
Recall that
\[
\aligned u  & = e^{-t(-\Delta)^{\a_1}}u_0 - u_1 + y,\\
b & = e^{-t(-\Delta)^{\a_2}}b_0 - b_1 + z.
\endaligned
\]
Due to (\ref{u1}) and (\ref{b1}) we have
\begin{align} \notag
u_1(x, t)= &\Bb_{\a_1}(e^{-t(-\Delta)^{\a_1}}u_0(x),e^{-t(-\Delta)^{\a_1}}u_0(x))\\
&-\Bb_{\a_1}(e^{-t(-\Delta)^{\a_1}}b_0(x),e^{-t(-\Delta)^{\a_1}}b_0(x)),\notag
\end{align}
\begin{align} \notag
b_1(x, t)= &\Bb_{\a_2}(e^{-t(-\Delta)^{\a_2}}u_0(x),e^{-t(-\Delta)^{\a_2}}b_0(x))\\
&-\Bb_{\a_2}(e^{-t(-\Delta)^{\a_2}}b_0(x),e^{-t(-\Delta)^{\a_2}}u_0(x)).\notag
\end{align}
By the fact that $k_i\cdot v=k'_i\cdot v'=0$ for all $i=1, 2, \ldots, r$,  it can be  immediately seen that
$$
\left(e^{-t(-\Delta)^{\a_1}}u_0\cdot\n\right)e^{-t(-\Delta)^{\a_1}}u_0 =\left(e^{-t(-\Delta)^{\a_1}}b_0\cdot\n\right)e^{-t(-\Delta)^{\a_1}}b_0 = 0, 
$$
hence $u_1\equiv 0$. Again since $k_i\cdot v'=0$ for all $i=1, 2, \ldots, r$ by (\ref{vk0}), it follows
$$
\left(e^{-t(-\Delta)^{\a_2}}b_0\cdot\n\right)e^{-t(-\Delta)^{\a_2}}u_0= 0, 
$$
hence 
\begin{equation} \label{new-b1}
b_1(x, t)= \Bb_{\a_2}(e^{-t(-\Delta)^{\a_2}}u_0(x),e^{-t(-\Delta)^{\a_2}}b_0(x)).
\end{equation}

Note that $\cos x\sin y=[\sin(x+y)-\sin(x-y)]/2$. Thanks to (\ref{eq:v}), (\ref{vk0}), (\ref{h-u0}), one can infer that 
\begin{equation}\label{u0interaction}
\begin{split}
&\left(e^{-t(-\Delta)^{\a_2}}u_0\cdot\n\right) e^{-t(-\Delta)^{\a_2}}b_0 \\
 =&-r^{-{\beta_1}-{\beta_2}}\sum_{i=1}^r\sum_{j=1}^r|k_i|^{\theta_1}|k'_j|^{\theta_2}e^{-(|k_i|^{2{\a_2}}+|k'_j|^{2{\a_2}})t}v'\cos(k_i\cdot x)\sin(k'_j\cdot x)\\ 
=&-\frac{r^{-{\beta_1}-{\beta_2}}}{2}\sum_{i=1}^r|k_i|^{\theta_1}|k'_i|^{\theta_2}e^{-(|k_i|^{2{\a_2}}+|k'_i|^{2{\a_2}})t}\sin(\eta\cdot x)v'\\ 
& -\frac{r^{-{\beta_1}-{\beta_2}}}{2}\sum_{i\neq j}^r|k_i|^{\theta_1}|k'_j|^{\theta_2}e^{-(|k_i|^{2{\a_2}}+|k'_j|^{2{\a_2}})t}\sin((k'_j-k_i)\cdot x)v'\\ 
& -\frac{r^{-{\beta_1}-{\beta_2}}}{2}\sum_{i=1}^r\sum_{j=1}^r|k_i|^{\theta_1}|k'_j|^{\theta_2}e^{-(|k_i|^{2{\a_2}}+|k'_j|^{2{\a_2}})t}\sin((k'_j+k_i)\cdot x)v'\\
:= &E_0+E_1+E_2.
\end{split}
\end{equation}

Since $\eta\cdot v'=0$, $(k'_j+k_i)\cdot v'=0$ and $(k'_j-k_i)\cdot v'=0$ for all $i,j$ (see \eqref{vk0}), it follows that  $E_{0}$, $E_{1}$ and $E_{2}$ are divergence free. 
Hence we can write 
\begin{equation}\label{dec:b1}
\begin{split}
b_1=&\int_0^te^{-(t-\tau)(-\Delta)^{\a_2}}E_{0}(\tau)d\tau+\int_0^te^{-(t-\tau)(-\Delta)^{\a_2}}E_{1}(\tau)d\tau\\
&+\int_0^te^{-(t-\tau)(-\Delta)^{\a_2}}E_{2}(\tau)d\tau := b_{10}+b_{11}+b_{12}.
\end{split}
\end{equation}

\begin{Lemma}\label{le:b10}
Let $b_{10}$ be as in (\ref{dec:b1}) and $s>0$. Assume $\theta_1+\theta_2=2\a_2$. Then
\begin{equation}\notag 
\begin{split}
&\|b_{10}(\cdot, t)\|_{\dot B_{\infty,\infty}^{-s}}\gtrsim r^{1-\beta_1-\beta_2},  \qquad |k_1|^{-2\a_2}\leq t\leq T,\\
&\|b_{10}(\cdot, t)\|_{\infty} \lesssim r^{1-\beta_1-\beta_2},  \qquad t>0.
\end{split}
\end{equation}
\end{Lemma}
\begin{proof}
Using (\ref{u0interaction}) and (\ref{dec:b1}) on can easily obtain
\begin{equation}\notag
\begin{split}
b_{10}=&-\frac{r^{-\beta_1-\beta_2}}{2}\int_0^t\sum_{i=1}^r|k_i|^{\theta_1}|k'_i|^{\theta_2}e^{-(|k_i'|^{2\a_2}+|k_i|^{2\a_2})\tau}e^{-|\eta|^{2\a_2}(t-\tau)}\sin(\eta\cdot x)v'd\tau\\
=&-\frac{r^{-\beta_1-\beta_2}}{2}\sin (\eta\cdot x)v'\sum_{i=1}^r|k_i|^{\theta_1}|k'_i|^{\theta_2}e^{-t}\frac{1-e^{-(|k_i'|^{2\a_2}+|k_i|^{2\a_2}-1)t}}{|k'_i|^{2\a_2}+|k_i|^{2\a_2}-1}\\
\sim &-\frac{r^{-\beta_1-\beta_2}}{2}\sin (\eta\cdot x)v'\sum_{i=1}^re^{-t}(1-e^{-|k_i|^{2\a_2}t}),
\end{split}
\end{equation}
where we used 
$$|k_i|^{\theta_1}|k'_i|^{\theta_2}\sim |k'_i|^{2\a_2}+|k_i|^{2\a_2}-1,$$
which holds due to assumption $\theta_1+\theta_2=2\a_2$. Therefore,
\[
\|b_{10}(\cdot,t)\|_{\dot B_{\infty,\infty}^{-s}}\gtrsim r^{-\beta_1-\beta_2}\cdot r\sup_{0<t<1}t^{\frac{s}{2\a}}e^{-|\eta|^{2\a}t}
\gtrsim r^{1-\beta_1-\beta_2},
\]
for $|k_1|^{-2\a_2}\leq t\leq T$ and $s>0$.
On the other hand,
\[
\|b_{10}(\cdot,t)\|_{\infty}\lesssim \frac{r^{-\beta_1-\beta_2}}{2}\cdot r\lesssim r^{1-\beta_1-\beta_2}, \qquad t>0.
\]
\end{proof}

\begin{Lemma}\label{le:b11}
Let $b_{11}$ and $b_{12}$ be as in (\ref{dec:b1}). Then
\begin{equation}\notag
\|b_{11}(\cdot, t)\|_{\infty}+\|b_{12}(\cdot, t)\|_{\infty}\lesssim r^{-\beta_1-\beta_2}t^{1-(\theta_1+\theta_2)/{2\a_2}}, \qquad t>0. 
\end{equation}
\end{Lemma}
\begin{proof}
Using (\ref{u0interaction}), (\ref{dec:b1}), and the fact that the fact that $\frac{1-e^{-x}}{x}$ is bounded for $x>0$, we can estimate
\begin{align}\notag
b_{11}=&\frac{r^{-\beta_1-\beta_2}}{2}\int_0^t\sum_{i\neq j}^r|k_i|^{\theta_1}|k'_j|^{\theta_2}e^{-(|k_i|^{2\a_2}+|k'_j|^{2\a_2})\tau}e^{-|k'_j-k_i|^{2\a_2}(t-\tau)}\\
&\cdot \sin((k'_j-k_i)\cdot x)v'd\tau\notag\\
\sim& \frac{r^{-\beta_1-\beta_2}}{2}\sum_{i=1}^r\sum_{j<i}|k_i|^{\theta_1}|k'_j|^{\theta_2}e^{-|k_i-k'_j|^{2\a_2}t}\frac{1-e^{-(|k_i|^{2\a_2}+|k'_j|^{2\a_2}-|k_i-k'_j|^{2\a_2})t}}{|k_i|^{2\a_2}+|k'_j|^{2\a_2}-|k_i-k'_j|^{2\a_2}}\notag\\
&\cdot\sin ((k'_j-k_i)\cdot x)v'\notag\\
\sim& \frac{r^{-\beta_1-\beta_2}}{2}\sum_{i=1}^r\sum_{j<i}|k_i|^{\theta_1}|k'_j|^{\theta_2}te^{-|k_i|^{2\a_2}t}
\sin ((k'_j-k_i)\cdot x)v'\notag.
\end{align}
Now thanks to  (\ref{kll1}) and (\ref{ksum2}) we have
\begin{align}\notag
\|b_{11}(\cdot, t)\|_{\infty}&\lesssim r^{-\beta_1-\beta_2}\sum_{i=1}^r\sum_{j<i}|k_i|^{\theta_1}|k'_j|^{\theta_2}te^{-|k_i|^{2\a_2}t}\\
&\lesssim r^{-\beta_1-\beta_2}\sum_{i=1}^r|k_i|^{\theta_1}|k'_i|^{\theta_2}te^{-|k_i|^{2\a_2}t}\notag\\
&\lesssim r^{-\beta_1-\beta_2}\sum_{i=1}^r|k_i|^{\theta_1+\theta_2}te^{-|k_i|^{2\a_2}t}\notag\\
&\lesssim r^{-\beta_1-\beta_2}t^{1-(\theta_1+\theta_2)/{2\a_2}}.\notag
\end{align}
Similarly, we have
\begin{align}\notag
b_{12}=&\frac{r^{-\beta_1-\beta_2}}{2}\int_0^t\sum_{i=1}^r\sum_{j=1}^r|k_i|^{\theta_1}|k'_j|^{\theta_2}e^{-(|k_i|^{2\a_2}+|k'_j|^{2\a_2})\tau}e^{-|k_i+k'_j|^{2\a_2}(t-\tau)}\\
&\cdot\sin((k_i+k'_j)\cdot x)v'd\tau\notag\\
=&\frac{r^{-\beta_1-\beta_2}}{2}
\sum_{i=1}^r\sum_{j=1}^r|k_i|^{\theta_1}|k'_j|^{\theta_2}e^{-(|k_i|^{2\a_2}+|k'_j|^{2\a_2})t}\frac{1-e^{-(|k_i+k'_j|^{2\a_2}-|k_i|^{2\a_2}-|k'_j|^{2\a_2})t}}{|k_i+k'_j|^{2\a_2}-|k_i|^{2\a_2}-|k'_j|^{2\a_2}}\notag\\
&\cdot\sin ((k_i+k'_j)\cdot x)v'\notag\\
\sim& r^{-\beta_1-\beta_2}\sum_{i=1}^r\sum_{j\leq i}|k_i|^{\theta_1}|k'_j|^{\theta_2}te^{-(|k_i|^{2\a_2}+|k'_j|^{2\a_2})t}\sin ((k_i+k'_j)\cdot x)v'\notag.
\end{align}
Thus,
\begin{align}\notag
\|b_{12}(\cdot, t)\|_{\infty}&\lesssim r^{-\beta_1-\beta_2}\sum_{i=1}^r\sum_{j\leq i}|k_i|^{\theta_1}|k'_j|^{\theta_2}te^{-|k_i|^{2\a}t}\\
&\lesssim r^{-\beta_1-\beta_2}\sum_{i=1}^r|k_i|^{\theta_1+\theta_2}te^{-|k_i|^{2\a}t}\notag\\
&\lesssim r^{-\beta_1-\beta_2}t^{1-(\theta_1+\theta_2)/{2\a_2}}.\notag
\end{align}
\end{proof}

\bigskip

\subsection{Analysis of $y$ and $z$}
\label{sec:yz}

In order to bound $y$ and $z$, we will utilize estimate (\ref{Bac}).

Recall that we have (see  Subsection \ref{sec:rewrite})
\bg\label{y-t}
y(t) = -\int\limits_{0}^{t}
e^{-(t-\tau) (-\Delta)^{\a_1}} [G_0 (\tau)+G_1 (\tau)+G_2 (\tau)]d\tau, \qquad t\in [0, T].
\ed 
\bg\label{z-t}
z(t) = -\int\limits_{0}^{t}
e^{-(t-\tau) (-\Delta)^{\a_2}} [K_0 (\tau)+K_1 (\tau)+K_2 (\tau)]d\tau, \qquad t\in [0, T].
\ed 


\begin{Lemma}\label{le:yinfty}
Let $\a_1, \a_2\in[1,\infty)$ and $\beta_1, \beta_2 \in (\frac{1}{3},\frac{1}{2})$. Assume additionally that 
\begin{equation}\label{para-theta}
\begin{cases}
\theta_1+\theta_2=2\a_2, \\
1\leq\theta_1\leq 4\a_2-\frac{\a_2}{\a_1}-1, \\
1\leq\theta_2\leq 4\a_1-\frac{\a_1}{\a_2}-1.
\end{cases}
\end{equation}
Then
\begin{equation}\label{yzinfty}
\begin{split}
\|y(t)\|_{\infty}+\|z(t)\|_{\infty}\lesssim &r^{1-\beta_1-2\beta_2}t^{1-\frac{1}{2\a_1}-\frac{\theta_2}{2\a_1}}+r^{2(1-\beta_1-\beta_2)}t^{1-\frac{1}{2\a_1}}\\
&+r^{1-2\beta_1-\beta_2}t^{1-\frac{1}{2\a_2}-\frac{\theta_1}{2\a_2}},
\end{split}
\end{equation}
for all $0\leq t\leq T$, provided  $T$ is small and $r$ is large enough.
\end{Lemma}
\begin{proof} 
Recall that $u_1=0$.
Using (\ref{eq:g}) and (\ref{y-t}) we obtain
\begin{equation}\notag
\begin{split}
\|y(t)\|_\infty\lesssim &
\|\Bb_{\a_1}(e^{-t(-\Delta)^{\a_1}}b_0, b_1(t))\|_\infty+\|\Bb_{\a_1}(b_1(t), b_1(t))\|_\infty\\
&+\|\Bb_{\a_1}(e^{-t(-\Delta)^{\a_1}}u_0, y(t))\|_\infty
+\|\Bb_{\a_1}(e^{-t(-\Delta)^{\a_1}}b_0, z(t))\|_\infty\\
&+\|\Bb_{\a_1}(b_1(t), z(t))\|_\infty
+\|\Bb_{\a_1}(y(t), y(t))\|_\infty+\|\Bb_{\a_1}(z(t), z(t))\|_\infty.
\end{split}
\end{equation}
Recall the following estimate for the Beta function:
\begin{align}\notag
\int_0^t(t-\tau)^{-\frac{1}{2{\a_1}}}\tau^{-\frac{\theta_2}{2\a_1}}d\tau=&t^{1-\frac{1}{2\a_1}-\frac{\theta_2}{2\a_1}}B(1-\frac{\theta_2}{2\a_1}, 1-\frac{1}{2\a_1})\\
\leq &Ct^{1-\frac{1}{2\a_1}-\frac{\theta_2}{2\a_1}},\notag
\end{align}
where $\a>1/2$ and $0<\frac{\theta_2}{2\a_1}<1$.
Now thanks to the bilinear estimate (\ref{Bac}), Lemmas \ref{le:u0infty}, \ref{le:b10}, and \ref{le:b11}, it follows that
\begin{align}\notag
\|\Bb_{\a_1}(e^{-t(-\Delta)^{\a_1}}b_0, b_1(t))\|_\infty
&\lesssim\int_0^t\frac{1}{(t-\tau)^{1/{2\a_1}}}\|e^{-\tau(-\Delta)^{\a_1}}b_0\|_{\infty}\|b_1(\tau)\|_{\infty}d\tau\notag\\
&\lesssim r^{1-\beta_1-2\beta_2}\int_0^t(t-\tau)^{-\frac{1}{2{\a_1}}}\tau^{-\frac{\theta_2}{2\a_1}}d\tau\notag\\
&\lesssim r^{1-\beta_1-2\beta_2}t^{1-\frac{1}{2\a_1}-\frac{\theta_2}{2\a_1}}\notag.
\end{align}
Due to the estimates obtained in previous two subsections, we also have the following:
\begin{equation}\notag
\|\Bb_{\a_1}(b_1(t), b_1(t))\|_\infty
\lesssim\int_0^t\frac{1}{(t-\tau)^{1/{2{\a_1}}}}\|b_1(\tau)\|_{\infty}^2d\tau\lesssim r^{2(1-\beta_1-\beta_2)}
t^{1-\frac{1}{2\a_1}},
\end{equation}
\begin{align}\notag
\|\Bb_{\a_1}(e^{-t(-\Delta)^{\a_1}}u_0, y(t))\|_\infty
&\lesssim\int_0^t\frac{1}{(t-\tau)^{1/{2\a_1}}}\|e^{-\tau(-\Delta)^{\a_1}}u_0\|_{\infty}\|y(\tau)\|_{\infty}d\tau\notag\\
&\lesssim r^{-\beta_1}\int_0^t(t-\tau)^{-\frac{1}{2\a_1}}\tau^{-\frac{\theta_1}{2\a_1}}d\tau\sup_{0<\tau<t}\|y(\tau)\|_{\infty}  \notag\\
&\lesssim r^{-\beta_1}t^{1-\frac{1}{2\a_1}-\frac{\theta_1}{2\a_1}}\sup_{0<\tau<t}\|y(\tau)\|_{\infty}\notag,
\end{align}
\begin{align}\notag
\|\Bb_{\a_1}(e^{-t(-\Delta)^{\a_1}}b_0, z(t))\|_\infty
&\lesssim\int_0^t\frac{1}{(t-\tau)^{1/{2\a_1}}}\|e^{-\tau(-\Delta)^{\a_1}}b_0\|_{\infty}\|z(\tau)\|_{\infty}d\tau\notag\\
&\lesssim r^{-\beta_2}\int_0^t(t-\tau)^{-\frac{1}{2\a_1}}\tau^{-\frac{\theta_2}{2\a_1}}d\tau\sup_{0<\tau<t}\|z(\tau)\|_{\infty}  \notag\\
&\lesssim r^{-\beta_2}t^{1-\frac{1}{2\a_1}-\frac{\theta_2}{2\a_1}}\sup_{0<\tau<t}\|z(\tau)\|_{\infty}\notag,
\end{align}
\begin{align}\notag
\|\Bb_{\a_1}(b_1(t), z(t))\|_\infty
&\lesssim\int_0^t\frac{1}{(t-\tau)^{1/{2\a_1}}}\|b_1(\tau)\|_{\infty}\|z(\tau)\|_{\infty}d\tau\notag\\
&\lesssim r^{1-\beta_1-\beta_2}\int_0^t(t-\tau)^{-\frac{1}{2\a_1}}d\tau\sup_{0<\tau<t}\|z(\tau)\|_{\infty}\notag\\
&\lesssim r^{1-\beta_1-\beta_2}t^{1-\frac{1}{2\a_1}}\sup_{0<\tau<t}\|z(\tau)\|_{\infty}\notag,
\end{align}
\begin{align}\notag
&\|\Bb_{\a_1}(y(t), y(t))\|_\infty+\|\Bb_{\a_1}(z(t), z(t))\|_\infty\\
\lesssim&\int_0^t\frac{1}{(t-\tau)^{1/{2\a_1}}}\left(\|y(\tau)\|_{\infty}^2+\|z(\tau)\|_{\infty}^2\right)d\tau\notag\\
\lesssim &t^{1-\frac{1}{2\a_1}}\left(\sup_{0<\tau<t}\|y(\tau)\|_{\infty}^2+\sup_{0<\tau<t}\|z(\tau)\|_{\infty}^2\right)\notag.
\end{align}
Therefore,
\begin{equation}\label{yinfty}
\begin{split}
\|y(t)\|_{\infty}
\lesssim & r^{1-\beta_1-2\beta_2}t^{1-\frac{1}{2\a_1}-\frac{\theta_2}{2\a_1}}+r^{2(1-\beta_1-\beta_2)}t^{1-\frac{1}{2\a_1}}\\
&+r^{-\beta_1}t^{1-\frac{1}{2\a_1}-\frac{\theta_1}{2\a_1}}\sup_{0<\tau<t}\|y(\tau)\|_{\infty}
+r^{-\beta_2}t^{1-\frac{1}{2\a_1}-\frac{\theta_2}{2\a_1}}\sup_{0<\tau<t}\|z(\tau)\|_{\infty}\\
&+r^{1-\beta_1-\beta_2}t^{1-\frac{1}{2\a_1}}\sup_{0<\tau<t}\|z(\tau)\|_{\infty}\\
&+t^{1-\frac{1}{2\a_1}}\left(\sup_{0<\tau<t}\|y(\tau)\|_{\infty}^2+\sup_{0<\tau<t}\|z(\tau)\|_{\infty}^2\right).
\end{split}
\end{equation}
By (\ref{eq:k}), (\ref{z-t}) and the fact $u_1=0$, we have
\begin{equation}\notag
\begin{split}
\|z(t)\|_\infty\lesssim &\|\Bb_{\a_1}(e^{-t(-\Delta)^{\a_2}}u_0, b_1(t))\|_\infty
+\|\Bb_{\a_2}(e^{-t(-\Delta)^{\a_2}}u_0, z(t))\|_\infty\\
&+\|\Bb_{\a_2}(e^{-t(-\Delta)^{\a_2}}b_0, y(t))\|_\infty
+\|\Bb_{\a_2}(b_1(t), y(t))\|_\infty
+\|\Bb_{\a_2}(y(t), z(t))\|_\infty.
\end{split}
\end{equation}
For $\a_2>1/2$ and $0<\frac{\theta_1}{2\a_2}, \frac{\theta_2}{2\a_2}<1$, a similar calculation shows that 
\begin{equation}\label{zinfty}
\begin{split}
\|z(t)\|_\infty\lesssim & r^{1-2\beta_1-\beta_2}t^{1-\frac{1}{2\a_2}-\frac{\theta_1}{2\a_2}}
+r^{-\beta_1}t^{1-\frac{1}{2\a_2}-\frac{\theta_1}{2\a_2}}\sup_{0<\tau<t}\|z(\tau)\|_\infty\\
&+r^{-\beta_2}t^{1-\frac{1}{2\a_2}-\frac{\theta_2}{2\a_2}}\sup_{0<\tau<t}\|y(\tau)\|_\infty
+r^{1-\beta_1-\beta_2}t^{1-\frac{1}{2\a_2}}\sup_{0<\tau<t}\|y(\tau)\|_\infty\\
&+t^{1-\frac{1}{2\a_2}}\sup_{0<\tau<t}\|y(\tau)\|_\infty\sup_{0<\tau<t}\|z(\tau)\|_\infty.
\end{split}
\end{equation}
Let $w(t)=\|y(t)\|_\infty+\|z(t)\|_\infty$ and $w=\sup_{0<\tau<t}\|y(\tau)\|_\infty+\sup_{0<\tau<t}\|z(\tau)\|_\infty$. Adding (\ref{yinfty}) and (\ref{zinfty}) yields
\begin{equation}\notag
\begin{split}
w(t)\lesssim &r^{1-\beta_1-2\beta_2}t^{1-\frac{1}{2\a_1}-\frac{\theta_2}{2\a_1}}+r^{2(1-\beta_1-\beta_2)}t^{1-\frac{1}{2\a_1}}+r^{1-2\beta_1-\beta_2}t^{1-\frac{1}{2\a_2}-\frac{\theta_1}{2\a_2}}\\
&+\left(r^{-\beta_1}t^{1-\frac{1}{2\a_1}-\frac{\theta_1}{2\a_1}}+
r^{-\beta_2}t^{1-\frac{1}{2\a_2}-\frac{\theta_2}{2\a_2}}+
r^{1-\beta_1-\beta_2}t^{1-\frac{1}{2\a_2}}\right)\sup_{0<\tau<t}\|y(\tau)\|_{\infty}\notag\\
&+\left(r^{-\beta_2}t^{1-\frac{1}{2\a_1}-\frac{\theta_2}{2\a_1}}+
r^{1-\beta_1-\beta_2}t^{1-\frac{1}{2\a_1}}+
r^{-\beta_1}t^{1-\frac{1}{2\a_2}-\frac{\theta_1}{2\a_2}}\right)\sup_{0<\tau<t}\|z(\tau)\|_\infty\notag\\
&+\left(t^{1-\frac{1}{2\a_1}}+t^{1-\frac{1}{2\a_2}}\right)w^2\notag\\
:=& C_1+C_2\sup_{0<\tau<t}\|y(\tau)\|_{\infty}+C_3\sup_{0<\tau<t}\|z(\tau)\|_{\infty}+C_4w^2\notag\\
\lesssim & C_1+(C_2+C_3) w+ C_4w^2\notag\\
= & C_1+\left (C_2+C_3+C_4w\right)w.
\end{split}
\end{equation}

We shall choose large enough $r$ and  small enough $T>0$,  so that 
\begin{equation}\label{para1}
C_2+C_3+C_1C_4\ll 1
\end{equation}
for $0\leq t\leq T$. Since $w(0)=0$, we have the following bound:
\bg\notag
\|w(t)\|_\infty\lesssim C_1\lesssim r^{1-\beta_1-2\beta_2}t^{1-\frac{1}{2\a_1}-\frac{\theta_2}{2\a_1}}+r^{2(1-\beta_1-\beta_2)}t^{1-\frac{1}{2\a_1}}+r^{1-2\beta_1-\beta_2}t^{1-\frac{1}{2\a_2}-\frac{\theta_1}{2\a_2}},
\ed
for all $0<t\leq T$. 

Indeed, since the powers of $t$ in $A$ are all nonnegative by the hypothesis of the lemma, it follows that
\begin{align}\notag
C_2+C_3+C_1C_4\lesssim & r^{1-\beta_1-\beta_2}T^{1-\frac{1}{2\a_2}}+r^{1-\beta_1-\beta_2}T^{1-\frac{1}{2\a_1}}\\
&+r^{2(1-\beta_1-\beta_2)}T^{1-\frac{1}{2\a_1}}\left(T^{1-\frac{1}{2\a_1}}+T^{1-\frac{1}{2\a_2}}\right)\notag.
\end{align}
 Let $T=r^{-\gamma}$. Then
\begin{equation}\label{para-C}
\begin{split}
C_2+C_3+C_1C_4\lesssim & r^{1-\beta_1-\beta_2-\gamma(1-\frac{1}{2\a_2})}+r^{1-\beta_1-\beta_2-\gamma(1-\frac{1}{2\a_1})}\\
&+r^{2(1-\beta_1-\beta_2)-\gamma(2-\frac{1}{\a_1})}+r^{2(1-\beta_1-\beta_2)-\gamma(2-\frac{1}{2\a_1}-\frac{1}{2\a_2})}.
\end{split}
\end{equation}
Now we choose $\gamma$ satisfying 
\bg\label{para-gamma}
\gamma>\frac{1-\beta_1-\beta_2}{1-1/(2\a_1)}, \qquad \gamma>\frac{1-\beta_1-\beta_2}{1-1/(2\a_2)}.
\ed
Then all the powers of $r$ are negative in (\ref{para-C}). Thus (\ref{para1}) is satisfied for $r$ large enough,
which proves the conclusion of the lemma.
\end{proof}

\subsection{Finishing the proof}
\label{subsec:end}

Here we  complete the proof of Theorem \ref{Mthm}.
Recall that  $u_0$ and $b_0$ are space-periodic and smooth. Then there exists $T^*>0$ and
a smooth space-periodic solution $(u(t), b(t))$ to \eqref{FMHD} on $[0,T^*)$
with $u(0)=u_0, b(0)=b_0$, so that either $T^*= +\infty$ or
\[
\limsup_{t\to T^*-} \left(\|u(t)\|_{\infty}+\|b(t)\|_{\infty}\right) = +\infty.
\]
Lemmas~\ref{le:b10},
\ref{le:b11}, and \ref{le:yinfty} imply that $T^* >T$. Thanks to
(\ref{b}), (\ref{ineq:imbed}), Lemmas \ref{le:u0infty}, \ref{le:b10}, \ref{le:b11}, and \ref{le:yinfty}, it follows that
\begin{equation}\label{ineq:blow-up}
\begin{split}
&\|b(\cdot, t)\|_{\dot{B}_{\infty,\infty}^{-s}}\\
\geq&\|b_{10}(\cdot, t)\|_{\dot{B}_{\infty,\infty}^{-s}} - \|b_{11}(\cdot, t)\|_{\infty}-\|b_{12}(\cdot, t)\|_{\infty}-\|e^{-t(-\Delta)^{\a_2}}b_0\|_{\infty}
- \|z(\cdot, t)\|_{\infty}  \\
 \gtrsim &r^{1-\beta_1-\beta_2}
 \left(1-r^{\beta_1-1}t^{-\frac{\theta_2}{2\a_2}}-r^{-\beta_2}t^{1-\frac{1}{2\a_1}-\frac{\theta_2}{2\a_1}} 
 -r^{1-\beta_1-\beta_2}t^{1-\frac{1}{2\a_1}}-r^{-\beta_1}t^{1-\frac{1}{2\a_2}-\frac{\theta_1}{2\a_2}}\right)\\
 \gtrsim &r^{1-\beta_1-\beta_2}\left(1-r^{\beta_1-1}|k_1|^{\theta_2}-r^{-\beta_2}T^{1-\frac{1}{2\a_1}-\frac{\theta_2}{2\a_1}}-r^{1-\beta_1-\beta_2}T^{1-\frac{1}{2\a_1}}\right),
\end{split}
\end{equation}
for $|k_1|^{-2\a_2}\leq t\leq T$. We will choose parameters so that 
\begin{equation}\label{small-p}
A:=r^{\beta_1-1}|k_1|^{\theta_2}+r^{-\beta_2}T^{1-\frac{1}{2\a_1}-\frac{\theta_2}{2\a_1}}+r^{1-\beta_1-\beta_2}T^{1-\frac{1}{2\a_1}}\leq1/2.
\end{equation}
Let $|k_1|=r^{\zeta}$, where $\zeta>0$, and $T=r^{-\gamma}$ as in Lemma \ref{le:yinfty}. Then
\[
A=r^{\beta_1-1+\zeta\theta_2}+r^{-\beta_2-\gamma(1-\frac{1}{2\a_1}-\frac{\theta_2}{2\a_1})}+r^{1-\beta_1-\beta_2-\gamma(1-\frac{1}{2\a_1})}.
\]
To make (\ref{small-p}) hold for large enough $r$, it is sufficient to choose $\zeta,\gamma$ such that 
\begin{equation}\label{para-gamma-theta}
\begin{cases}
& 0<\zeta<\frac{1-\beta_1}{\theta_2}, \\
& \frac{1-\beta_1-\beta_2}{1-1/(2\a_1)} < \gamma <2\a_2\zeta,\\
& 1-\frac{1}{2\a_1}-\frac{\theta_2}{2\a_1}\geq 0.
\end{cases}
\end{equation}
Moreover, due to $\gamma<2\a_2\zeta$ in (\ref{para-gamma-theta}), we have that that  $|k_1|^{-2\a_2}< T$, which is required in Lemma \ref{le:b10}.
We verify that there exist $\beta_1, \beta_2, \theta_1$ and $\theta_2$ such that the assumption 
(\ref{para-theta}) in Lemma \ref{le:yinfty} and conditions (\ref{para-gamma}) (\ref{para-gamma-theta}) are compatible. Indeed, one can take $\beta_1=\beta_2=\frac{1}{2}-\epsilon$ for $0<\epsilon<\frac{1}{6}$. Since $\a_1, \a_2\geq 1$, it follows
$$
\frac{1-\beta_1-\beta_2}{1-1/(2\a_1)}\leq 4\epsilon, \qquad 
\frac{1-\beta_1-\beta_2}{1-1/(2\a_2)}\leq 4\epsilon.
$$
On the other hand,  due to the relation $\theta_1+\theta_2=2\a_2$ in (\ref{para-theta}), we have
\begin{equation}\notag
2\a_2\frac{1-\beta_1}{\theta_2}> \frac{1}{2}+\epsilon.
\end{equation}
Thus, the first two conditions in (\ref{para-gamma-theta}) are compatible. A straightforward computation shows that  all the other conditions among (\ref{para-theta}), (\ref{para-gamma}) and (\ref{para-gamma-theta}) are also compatible.\\

Now given any $\delta>0$, we now can take $r$ large enough so that
\bg\notag
r^{1-\beta_1-\beta_2}\gtrsim \frac{1}{\delta}. 
\ed
Then (\ref{ineq:blow-up}) and (\ref{small-p}) yield
\begin{equation}\notag
\|u(\cdot, T)\|_{\dot{B}_{\infty,\infty}^{-s}} \gtrsim r^{1-\beta_1-\beta_2}\gtrsim \frac{1}{\delta}.
\end{equation}
Finally, we have 
\bg\notag
\|u_0\|_{\dot{B}_{\infty,\infty}^{-\theta_1}}\lesssim r^{-\beta_1}\lesssim \delta, \qquad
\|b_0\|_{\dot{B}_{\infty,\infty}^{-\theta_2}}\lesssim r^{-\beta_2}\lesssim \delta.
\ed
due to Lemma \ref{le:u0}, which competes the proof of Theorem \ref{Mthm}.

\end{document}